\documentclass{amsart}
\usepackage[utf8]{inputenc}
\usepackage[utf8]{inputenc}
\usepackage[english]{babel}
\usepackage{graphicx}
\usepackage{mathrsfs}
\usepackage{amssymb}
\usepackage{framed}
\usepackage{amsmath,amsthm,amssymb,amstext,amsfonts}
\usepackage{appendix}
\usepackage{upgreek}
\usepackage[top=1 in,bottom=1 in, left=1 in, right=1 in]{geometry}
\usepackage{indentfirst}
\usepackage{tikz,tikz-cd}
\usepackage{braket}
\usepackage{color}

\DeclareMathOperator{\Ker}{Ker}

\newcommand*{\dif}{\mathop{}\!\mathrm{d}}

\theoremstyle{definition}
\newtheorem{theorem}{Theorem}[section]

\newtheorem{lemma}{Lemma}[section]
\newtheorem{conjecture}{Conjecture}[section]

\newtheorem{question}{Question}[section]
\newtheorem{remark}{Remark}[section]

\title{ A Note on K\"ahler-Ricci flow on Fano threefolds}
\author{Minghao Miao and Gang Tian}
\address{Minghao Miao: Department of Mathematics, Nanjing University, Nanjing 210093, China}
\email{minghao.miao@smail.nju.edu.cn}
\address{Gang Tian: BICMR and SMS, Peking University, Beijing 100871, China}
\email{gtian@math.pku.edu.cn}
\date{\today}

\begin{document}

\begin{abstract}
    In this note, we show that the solution of K\"ahler-Ricci flow on every Fano threefold from the family No.2.23 in the Mori-Mukai's list develops type II
    singularity. In fact, we show that no Fano threefold from the family No.2.23 admits Kähler-Ricci soliton and the Gromov-Hausdorff limit of the K\"ahler-Ricci flow must be a singular $\mathbb{Q}$-Fano variety. This gives new examples of Fano manifolds of the lowest dimension on which Kähler-Ricci flow develops type II singularity.
\end{abstract}

\maketitle

\section{Introduction}

Ricci flow, which was introduced by R. Hamilton in \cite{Hamilton}, has many developments in geometry over the past 40 years. It preserves the K\"ahlerian condition. When restricted to K\"ahler manifolds, it is referred to  as Kähler-Ricci flow. Let $X$ be a Fano manifold, that is, a compact Kähler manifold with positive first Chern class, we consider the following normalized Kähler-Ricci flow:
\begin{equation}\label{eq: KRF}
\begin{cases}
    \frac{\partial \omega (t)}{\partial t} &= -\text{Ric}\left(\omega(t)\right)+\omega(t) \\
    \omega(0) & = \omega_0
\end{cases}
\end{equation}
where $\omega_0$ and $\omega(t)$ denote the Kähler forms of initial Kähler metric $g_0$ and the solution $g(t)$ of Ricci flow, respectively.

It was shown in \cite{Cao} that when $\omega_0$ represents $2\pi c_1(X)$, equation (\ref{eq: KRF}) has a global solution $\omega(t)$ for all $t \geq 0$. Then it is natural to ask what is the limiting behavior of $\omega(t)$ as $t \rightarrow \infty$. Let us first recall the definition of Kähler-Ricci soliton, which is a generalization of Kähler-Einstein metrics. A Kähler metric $\omega$ is a (shrinking) Kähler-Ricci soliton if for some holomorphic vector field $\xi$ on Fano manifold $X$, we have:
$$
\text{Ric}(\omega) - \omega = L_{\xi} \omega
$$
where $L_{\xi} \omega$ is taking Lie derivative along $\xi$. When $\xi =0$, it reduces to the Kähler-Einstein metric.
It was proved in \cite{TianZhu07}, \cite{TianZhu15},\cite{TZZZ} and \cite{DerSze} that if $X$ admits a Kähler-Ricci soliton  $\omega_{\text{KRS}}$ and $\omega_0$ represents $2\pi c_1(X)$, then as $t$ goes to $\infty$, the solution $\omega(t)$ of the normalized Kähler-Ricci flow (\ref{eq: KRF}) converges to $\omega_{\text{KRS}}$ up to the action of the automorphism group of $X$.  It remains to see what we can expect if $X$ does not admit any Kähler-Ricci soliton. In this situation, the complex structures will jump as $t \rightarrow \infty$ and it was conjectured in \cite{Tian97}, also referred to as Hamilton-Tian conjecture, that:

\begin{conjecture}\label{thm:Hamilton-Tian}
For any global solution $\omega(t)$ of (\ref{eq: KRF}) as above, any sequence $(X,\omega(t))$ along Kähler-Ricci flow contains a subsequence converging to a $\mathbb{Q}$-Fano variety $(X_{\infty},\omega_{\infty})$ in the Gromov-Hausdorff topology, and $(X_{\infty},\omega_{\infty})$ admits a smooth shrinking Kähler-Ricci soliton outside the singular set $S$ of $X_\infty$ which is closed and of Hausdorff codimension at least 4. Moreover, this subsequence of $\left(X,\omega(t)\right)$ converges locally to $(X_{\infty},\omega_{\infty})$ in the Cheeger-Gromov topology.
\end{conjecture}

This conjecture has been first solved for dimension less than or equal to 3 (\cite{TianZhang}) and subsequently for higher dimensions (\cite{Bamler},\cite{ChenWang}, \cite{wangzhu}, \cite{BLXZ}) by using different methods. The key of these methods is to establish certain compactness for K\"ahler-Ricci flow which leads to the partial $C^0$-estimate along the flow. We would also like to mention that a generalized version of the Hamilton-Tian conjecture was proved in \cite{Bamler} and an algebraic proof of the conjecture was given in \cite{BLXZ} by using the work of \cite{HanLia}.

In \cite{TianZhang}, assuming the compactness mentioned above, the second author and Zhang established the partial $C^0$-estimate for the Kähler-Ricci flow and proved that the Gromov-Hausdorff limit of Kähler-Ricci flow is a normal projective variety whose singular set coincides with $S$. One can show that the Gromov-Hausdorff limit coincides with the algebro-geometric limit in this situation, that is, if $X_i$ converges to $X_{\infty}$ in the Gromov-Hausdorff topology, then $X_i$ and $X_{\infty}$ can be realized as fibers in a flat family. In \cite{ChenSunWang}, they describe a "two-step degeneration" picture of the above process and establish the uniqueness of Gromov-Hausdorff limit $X_{\infty}$. Han and Li further confirmed in  \cite{HanLib} that  the Gromov-Hausdorff limit $X_{\infty}$ depends only on the algebraic structure of $X$ and is independent of the choice of the initial metric of the flow.

We now turn to a brief discussion of the singularity formation of Kähler-Ricci flow. We recall that a solution $\omega(t)$ of Kähler-Ricci flow is of type I if the curvature of $\omega(t)$ is uniformly bounded, otherwise, we call $\omega(t)$ a solution of type II. There was a folklore conjecture closely related to Conjecture \ref{thm:Hamilton-Tian} stating that the Gromov-Hausdorff limit $X_\infty$ is always smooth, i.e., the singular set $S$ is always empty. This is equivalent to saying that Kähler-Ricci flow has no type $\rm II$ solution.  The folklore conjecture was first disproved in \cite{LiTianZhu} by considering Fano compactifications of semisimple Lie groups \footnote{More counterexamples on group compactifications were later found in \cite{Korea}.}. However, the lowest dimension among these group compactifications is 6. Naturally, one is led to wonder whether there exist examples of Fano manifold of lower dimensions on which Kähler-Ricci flow develops singularities of type II. Since the folklore conjecture holds for complex dimension two or less, a natural question is whether the lowest possibility is $3$.

In this short note, we answer this question by giving a family of examples of Fano threefolds that have type $\rm II$ solutions for the Kähler-Ricci flow. Here is our main theorem:

\begin{theorem}\label{thm: mainthm}
Any Fano threefold $X$ from family No. 2.23 in  Mori-Mukai's list has type $\rm II$ solutions of the normalized Kähler-Ricci flow. Namely, the Gromov-Hausdorff limit of $X$ along the Kähler-Ricci flow is a singular $\mathbb{Q}$-Fano variety.
\end{theorem}

Theorem \ref{thm: mainthm} will be deduced from the following theorem:
\begin{theorem}\label{thm: lemmathm}
Any Fano threefold $X$ from the family No. 2.23 in the Mukai-Mori's list does not admit Kähler-Ricci soliton.
\end{theorem}
This theorem will be proved by showing that the automorphism group of $X$ is finite (see Lemma \ref{lem:auto}) and $X$ is K-unstable (see Lemma \ref{lem: Kunstable}). In \cite{Delcroix}, Delcroix constructed examples of K-unstable Fano manifolds by blowing up quadrics along lower dimensional linear subquadrics and showed that they do not admit any Kähler-Ricci soliton. The lowest dimension of these examples of Delcroix is $5$, so he further asked for examples of Fano threefolds or fourfolds which are K-unstable and do not admit Kähler-Ricci soliton. The Theorem \ref{thm: lemmathm} gives an affirmative answer to Question 1.3 in \cite{Delcroix}.

\begin{question}(\cite{Delcroix}, Question 1.3) Does there exist examples of Fano threefolds with no
Kähler–Ricci solitons that are K-unstable? What about Fano fourfolds?

\end{question}

We leave the following question for future study:
\begin{question}
Suppose $X$ is a Fano threefold from the family No 2.23, can we find the Gromov-Hausdorff limit $X_{\infty}$ explicitly? Are members from the family 2.23 the only examples of Fano threefolds on which Kähler-Ricci flow has solutions of type $\rm II$?
\end{question}


\section{Family No 2.23 of Fano threefolds}

Due to the work of Iskovskikh-Mori-Mukai (see \cite{MukaiMori},\cite{Fano5}), we know every smooth Fano threefold belongs to one of the 105 deformation families.
Here is the description of the family No. 2.23 in the Mori-Mukai's list of smooth Fano threefold (see \cite{MukaiMori}): Any $X$ in this family  is obtained by blowing up a smooth quadric $Q \subset \mathbb{P}^4$ along a smooth irreducible curve $C$ which is an intersection of $H \in |\mathcal{O}_{Q}(1)|$ and $Q' \in |\mathcal{O}_{Q}(2)|$, note that $C$ is an elliptic curve of degree 4. There are two subfamilies:
\begin{itemize}
    \item Subfamily (a): the hypersurface $H \in |\mathcal{O}_{Q}(1)|$ is smooth.
    \item Subfamily (b): the hypersurface $H \in |\mathcal{O}_{Q}(1)|$ is singular, but its singular set has an empty intersection with $Q\cap Q'$.
\end{itemize}
\subsection{Automorphism groups of threefolds from family No. 2.23 are finite}
The following result was already known in \cite{CPS}, and we include a proof for the reader's convenience. This proof is more direct than the original one.

\begin{lemma}\label{lem:auto}
For each $X$ from family No. 2.23 , the automorphism group of $X$ is a finite group.
\end{lemma}
\begin{proof}
To begin with, we note that $\text{Aut}(X) = \text{Aut}(Q,C)$. Following the proof in  \cite{CPS}, we consider the following exact sequence:
$$
\begin{tikzcd}
0 \arrow[r] & \text{Ker}_H \arrow[r, "i"] & {\text{Aut}(Q,C)} \arrow[r, "r"] & {\text{Aut}(H,C)},
\end{tikzcd}
$$
where
\begin{align*}
    \text{Aut}(Q,C) &= \{\phi \in \text{Aut}(Q)| \ \phi(C)=C \},\\
    \text{Aut}(H,C) &= \{\phi \in \text{Aut}(H)| \ \phi(C)=C \},\\
    \text{Ker}_H&=\{\phi \in \text{Aut}(Q,C)| \ \forall x \in H, \phi(x) = x\},
\end{align*}
    $r: \text{Aut}(Q,C) \rightarrow \text{Aut}(H,C), \ \phi \mapsto \phi|_H$ is the restriction map and $i$ is the natural inclusion.

   We first show $r$ is well-defined. Observe that: (1) $\text{Aut}(Q,C) \subset \text{Aut}(\mathbb{P}^4)=\text{PSL}(5,\mathbb{C})$
   which is a linear algebraic group; (2) There are a hyperplane $\bar H$ and a quadratic $\bar Q'$ of $\mathbb{P}^4$ such that $H = \bar H\cap Q$ and $Q'=\bar Q'\cap Q$.\footnote{For H, such a lifting $\bar H$ is unique.}

   It follows from (1) that any $\phi\in \text{Aut}(Q,C)$ maps hyperplane $\bar H$ to another hyperplane $\bar H':=\phi(\bar{H})$. We want to show $\bar{H}'= \bar{H}$. If not, we assume $\bar H'\not= \bar H$, then $S=\bar H\cap \bar H'$ is 2-dimensional subspace. Since a smooth quadric of dimension 3 contains no linear spaces of dimension strictly greater than 1 (see, for example,   section 6.1 of \cite{griffithsharris}), we conclude that $C'=S\cap \bar Q'$ is a curve of degree 2, in particular, it is a rational curve and different from $C$.
   Note $C = Q \cap \bar Q' \cap \bar H \cap \bar H'\subset C'$ which implies $C=C'$, so we get a contradiction, consequently,
    we must have $\bar H = \bar H'$, it follows that $\phi|_H$ is well-defined.

    Next, we claim both $\text{Ker}_H$ and $\text{Aut}(H,C)$ are finite groups. Then by the exactness, we conclude $\text{Aut}(Q,C)$ is a finite group.

 To show $\text{Ker}_H$ is a finite group, after changing the coordinate, we have two cases to check:
 \begin{itemize}
     \item Case 1: If $H$ is smooth, we can choose a coordinate such that $H = V(x_0)$ and $Q = V(x_0^2+x_1^2+\cdots + x_4^2)$;
     \item Case 2: If $H$ is singular, we can choose a coordinate such that $H = V(x_0)$ and $Q = V(x_0x_1+x_2^2+\cdots+x_4^2)$.
 \end{itemize}

 Suppose $\phi \in \Ker_H\subset \text{PGL}(5,\mathbb{C})$, then by definition, for $\  x=[0:x_1:x_2:x_3:x_4] \in H, $ we get $\phi(x) = x$, hence $\phi$ must be of the form:
    $$
    B = \begin{pmatrix}
    a & b & c & d & e\\
    f & \lambda & 0 & 0 & 0 \\
    g & 0 & \lambda & 0 & 0\\
    h & 0 & 0 & \lambda & 0\\
    i& 0& 0 & 0 & \lambda\\
    \end{pmatrix} \in \text{PGL}(5,\mathbb{C})
    $$

  For Case 1, since $\phi \in \text{Aut}(Q)$, there exists $\mu \in \mathbb{C}^*$ such that $B \cdot B^T = \mu\cdot I_{5}$. Hence, we conclude:
    $$
    \begin{cases}
    a^2+b^2+c^2+d^2+e^2 = \mu\\
    af+b\lambda=ag+c\lambda=ah+d\lambda=ai+e\lambda=0 \\
    f^2+\lambda^2=g^2+\lambda^2=h^2+\lambda^2=i^2+\lambda^2=\mu \\
    fg=fh=fi=gh=gi=hi=0 \\

    \end{cases}
    $$

It is easy to show that $b=c=d=e=f=g=h=i=0, \ a^2=\lambda^2=\mu$. Therefore, we deduce $\text{Ker}_H \cong \mathbb{Z}/2\mathbb{Z}$ in Case 1.

 For Case 2, since $\phi \in \text{Aut}(Q)$, we have $B\cdot \begin{pmatrix}
 0 & \frac{1}{2} & 0 & 0 & 0\\
 \frac{1}{2} & 0 & 0 &0 & 0\\
 0 & 0 & 1&0&0\\
 0&0&0&1&0\\
 0&0&0&0&1\\
 \end{pmatrix}\cdot B^T = \mu \cdot \begin{pmatrix}
 0 & \frac{1}{2} & 0 & 0 & 0\\
 \frac{1}{2} & 0 & 0 &0 & 0\\
 0 & 0 & 1&0&0\\
 0&0&0&1&0\\
 0&0&0&0&1\\
 \end{pmatrix}$ for some $\mu \in \mathbb{C}^*$, we arrive at: $$
 \begin{cases}
 ab + c^2 + d^2 +e^2 = \lambda\cdot f = 0\\
 bf+a\lambda = \mu \\
 \frac{1}{2}bg+c\lambda = \frac{1}{2}bh + d\lambda = \frac{1}{2}bi+e\lambda =0\\
 \lambda g = \lambda h =\lambda i =0
 \end{cases}
 $$
 Hence, $b=c=d=e=f=g=h=i=0$ and $a = \lambda, \lambda^2 =\mu$, so $\Ker_H=\text{id}$ in Case 2.

 To show $\text{Aut}(H,C)$ is a finite group:  after intersecting a hyperplane in $\mathbb{P}^4$, we can reduce to the case for quadric surfaces $Q_1,Q_2 \subset \mathbb{P}^3$  and let $C =  Q_1 \cap Q_2$, and it suffices to show $\text{Aut}(Q_1, C)$ is a finite group. Since $C$ is smooth, then by the result of \cite{Miles} (see page 36), there exists a coordinate such that $Q_1$ and $Q_2$ can be simultaneously diagonalized:

    \begin{align*}
     Q_1 &= V(x_0^2+x_1^2+x_2^2+x_3^2) \\
     Q_2 &=
     V(x_0^2+\lambda_1 x_1^2+\lambda_2 x_2^2 + \lambda_3 x_3^2)
    \end{align*}
   where $1,\lambda_1,\lambda_2,\lambda_3$ are pairwise disjoint. Just denote $F_1 = x_0^2+x_1^2+x_2^2+x_3^2$ and $F_2 = x_0^2 + \lambda_1 x_1^2 + \lambda_2 x_2^2 + \lambda_3 x_3^2$, and we can rewrite $Q_1 = V(F_1), Q_2 = V(F_2)$. Take any automorphism $\phi \in \text{Aut}(Q_1,C)\subset \text{Aut}(\mathbb{P}^3)$. We have $\phi \cdot Q_1 = Q_1$ and $\phi \cdot Q_2 = V(F)$ where $F$ is a homogeneous polynomial of degree two. Note that $ C = V(F_1,F_2) $. From $C \subseteq\phi\cdot Q_2$, we have  $(F)\subseteq \sqrt{(F_1,F_2)}=(F_1,F_2)$ since $C$ is reduced. Then there exists $c,d \in \mathbb{C}$ such that $F = c \cdot F_1 + d\cdot F_2$ 
  because the degree of $F$ is two. We denote the matrix corresponding to automorphism $\phi$ by $P \in \text{PGL}(4,\mathbb{C})$. And let $B = \begin{pmatrix}
   1 & 0 & 0 & 0\\
   0 & \lambda_1 & 0 & 0\\
   0 & 0 & \lambda_2 & 0\\
   0 &0 &0& \lambda_3\\
   \end{pmatrix}$,
   then we have matrix equations:
   $$
   \begin{cases}
   P P^T & = I_4\\
   PBP^T & = c I_4 + dB\\
   \end{cases}
   $$
   Note $\text{Tr}(PBP^T)=\text{Tr}(B) = 4c + d \cdot \text{Tr}(B)$, then $c = \frac{1-d}{4}\text{Tr}(B)$.

   Let $\tilde{F}_2=F_2-\frac{1}{4}\text{Tr}(B)F_1$, and $\tilde{Q}_2 = V(\tilde{F}_2)$, then $\phi\cdot Q_1 = Q_1, \ \phi \cdot \tilde{Q}_2 = \tilde{Q}_2$. And $C = Q_1 \cap Q_2 = Q_1 \cap \tilde{Q}_2$. Namely, we finally reduce to find solutions for $P \in \text{GL}(4,\mathbb{C})$ satisfying:
   $$
   \begin{cases}
   P P^T & = I_4\\
   PBP^T & = B\\
   \end{cases}
   $$
   Since $1,\lambda_1,\lambda_2,\lambda_3$ are pairwise disjoint, we conclude that $P$ must be a diagonal matrix with diagonal entry $\pm 1$, which shows
    $\text{Aut}(H,C)$ is a finite group.

\end{proof}

\subsection{Threefolds from family No.2.23 are K-unstable}
In this subsection, we will show all members from family No. 2.23 is K-unstable by computing $\beta$ invariant (For a complete treatment of $\beta$ invariant, we refer our reader to section 3 of \cite{XuSurvey} or chapter 1 of \cite{FanoProject}). This result was already known by Fujita (\cite{Fujita}, Lemma 9.9, see also \cite{FanoProject} Lemma 3.7.4). For the reader's convenience, we also include a proof here.
\begin{lemma}\label{lem: Kunstable}
For every Fano threefold $X$ from family No. 2.23, $X$ is divisorially unstable, namely, there exists a divisor $E$ on $X$ such that $\beta_X(E) < 0$. In particular, this implies $X$ is K-unstable, so $X$ does not admit Kähler-Einstein metric.
\end{lemma}
\begin{proof}
Let $A$ be a hyperplane section of $Q$, then by adjunction formula $K_Q \sim -3A$ and $K_X -\pi^* K_Q =E$, where $E$ is the exceptional divisor, so $-K_X = 3\pi^*A -E$. Let $\tilde{H}$ be the strict transform of $H$, then $\pi^* H = \tilde{H}+E \sim \pi^*A$. For $t \geq 0$, consider $-K_X -t\cdot \tilde{H}=(3-t)\pi^*A +(t-1)E$. Note for $0\leq t \leq 3$, $-K_X - t\cdot \tilde{H}$ is a big divisor. To check nefness:
$$
\begin{cases}
(-K_X - t\cdot \tilde{H})\cdot f &= 1-t \geq 0\\
(-K_X - t\cdot \tilde{H})\cdot \pi^* L &= 3-t \geq 0
\end{cases}
$$
where $f\subset E$ is a line in exceptional surface and $L$ is a line in quadric threefold $Q$ which does not intersect with $C$. Hence, $-K_X - t\cdot \tilde{H}$ is nef if and only if $0\leq t \leq 1$. If $1<t \leq 3$, Zariski decomposition of $-K_X -t\cdot \tilde{H}$ is given by $-K_X-t\tilde{H}=(3-t)\cdot \pi^*A +(t-1)\cdot E$ where $(3-t)\cdot \pi^*A $ is the positive part and $(t-1)\cdot E$ is the negative part. By the formula in \cite{Fano5} (see Lemma 2.2.14), we have
\begin{align*}
    (-K_X - t\cdot \tilde{H})^3 & = [(3-t)\cdot \pi^*A + (t-1)\cdot E]^3\\
    &=(3-t)^3\cdot A^3 + 0 + 3(3-t)(t-1)^2\pi^*A\cdot (-\pi^*C+\text{deg}(N_{C/Q})f)-(t-1)^3c_1(N_{C/Q})\\
    &=(3-t)^3\cdot A^3 - 3(3-t)(t-1)^2 A\cdot C - (t-1)^3 (3A\cdot C)\\
    &=-2(t^3+3t^2+3t-15)
\end{align*}
Thus,
$$
\text{vol}(-K_X - t\cdot \tilde{H}) =
\begin{cases}
-2(t^3+3t^2+3t-15) & 0\leq t \leq 1\\
2(3-t)^3 & 1<t \leq 3
\end{cases}
$$
Therefore,
\begin{align*}
\beta_X(\tilde{H})&=A_X(\tilde{H})-S_X(\tilde{H})\\
&=1 - \frac{1}{30}\int_0^1 -2(t^3+3t^2+3t-15)\dif t - \frac{1}{30} \int_1^3 2(3-t)^3 \dif t\\
&=-\frac{1}{12}<0
\end{align*}
Hence $X$ is divisorially unstable, so $X$ is K-unstable by Fujita-Li's valuation criterion (\cite{LiChi},\cite{Fujita}), then $X$ does not admit Kähler-Einstein metric by Yau-Tian-Donaldson theorem.
\end{proof}

\begin{remark}
The threefolds in Family No.2.23 provide more counterexamples to the following conjecture: {\it If a compact Fano manifold has no nontrivial holomorphic vector fields, then it admits a Kähler-Einstein metric}. The first counterexample was constructed by the second author in \cite{Tian97} by deforming the Mukai-Umemura threefold.
\end{remark}

\section{Proof of Theorem \ref{thm: lemmathm} and Theorem \ref{thm: mainthm}}
We will first prove Theorem \ref{thm: lemmathm} which is restated as follows:

\begin{theorem}\label{thm:NoKRS}
For all Fano threefolds $X$ from family No 2.23, $X$ does not admit Kähler-Ricci soliton.
\end{theorem}
\begin{proof}
This theorem directly follows from Lemma \ref{lem:auto} and Lemma \ref{lem: Kunstable}. We prove by contradiction, assuming that $X$ admits a Kähler-Ricci soliton, namely, there exists a Kähler metric $\omega$ in Kähler class $2\pi c_1(X)$ and a holomorphic vector field $\xi$ such that $\text{Ric}(\omega)-\omega = L_{\xi} \omega$. By Lemma \ref{lem:auto}, all holomorphic vector fields on $X$ must be trivial, so $\omega$ becomes a Kähler-Einstein metric, which contradicts 
that $X$ does not admit Kähler-Einstein metric from Lemma \ref{lem: Kunstable}. Therefore, we conclude that $X$ does not admit any Kähler-Ricci soliton.
\end{proof}

Next, we prove Theorem \ref{thm: mainthm} which is restated as follows:

\begin{theorem}
Any Fano threefold $X$ from Family No. 2.23 in  Mori-Mukai's list has type $\rm II$ solutions for the normalized Kähler-Ricci flow.
Namely, the Gromov-Hausdorff limit of Kähler-Ricci flow on $X$ is a singular $\mathbb{Q}$-Fano variety.
\end{theorem}

\begin{proof}
We prove by contradiction again, assuming that the Gromov-Hausdorff limit $X_{\infty}$ is smooth, then there exists a smooth deformation $\pi:\mathcal{X}\rightarrow \Delta$ such that $\pi$ is a smooth morphism and general fiber $\pi^{-1}(t)\cong X$ for $t\neq 0$ and $\pi^{-1}(0) = X_{\infty}$. Since
$\pi$ is a smooth morphism, by Ehresmann's theorem, we conclude $b_2(X)=b_2(X_{\infty})=2,b_3(X)=b_3(X_{\infty})=2$ and $\text{Vol}(X)=\text{Vol}(X_{\infty})=30$, from the geography of smooth Fano threefolds, we conclude that $X_{\infty}$ still belongs to Family No. 2.23. Since $X_\infty$ is smooth, it has been known for long that $X_{\infty}$ admits a Kahler-Ricci soliton, which contradicts with Theorem \ref{thm:NoKRS}. Therefore, the limit $X_{\infty}$ must be singular.
\end{proof}

\bibliographystyle{alpha}
\bibliography{ref}

\end{document}